\documentclass[11pt,a4paper]{article}

\usepackage{inputenc}
\usepackage{amsmath}
\usepackage{bm}
\usepackage{bbold}
\usepackage{amsthm}

\usepackage{hyperref}

\title{Unbiased Gradient Estimation in Queueing Networks with Parameter-Dependent Routing\thanks{Proc. Intern. Conf. on Control and Information 1995 / Ed. by Wing Shing Wong. The Chinese University Press, Hong Kong, 1995. P. 351-356.}
}

\author{Nikolai K. Krivulin\thanks{Faculty of Mathematics and Mechanics, St.~Petersburg State University, Bibliotechnaya Sq.2, Petrodvorets, St.~Petersburg, 198904 Russia}}

\date{}

\newtheorem{theorem}{Theorem}

\newenvironment{algorithm}[1]{\vspace{2ex}{\sc Algorithm #1.}\bf
                                \begin{tabbing}}{\end{tabbing}}

\begin{document}

\maketitle

\begin{abstract}
A stochastic queueing network model with parameter-dependent service times and
routing mechanism, and its related performance measures are considered. An
estimate of performance measure gradient is proposed, and rather general
sufficient conditions for the estimate to be unbiased are given. A gradient
estimation algorithm is also presented, and its validity is briefly discussed.
\\

\end{abstract}

\section{Introduction}

The evaluation of performance measure gradient presents one of the main issues
of analysis of queueing network performance. Except in a few particular
models, there are generally no closed-form representations as functions of
network parameters available for performance measures and their gradients.
In this situation, one normally applies the Monte Carlo approach to estimate
gradient of network performance measures.

In the last decade, infinitesimal perturbation analysis (IPA)
\cite{HoYC83,HoYC87,Suri89} has received wide acceptance in queueing system
performance evaluation as an efficient technique underlying the calculation of
gradient estimates as well as the examination of their unbiasedness.
Specifically, this technique was employed in \cite{CaoX91} to calculate
gradient estimates in closed networks with an ordinary probabilistic routing
mechanism and general service time distributions. An extension of IPA,
smoothed perturbation analysis, has been applied in \cite{Gong92} to the
development of asymptotically unbiased gradient estimates in queueing networks
with parameter-dependent routing.

Another approach based on the analysis of algebraic representation of queueing
system dynamics and their performance has been implemented in
\cite{Kriv90a,Kriv90b,Kriv93}. This approach offers a convenient and unified
way of analytical study of gradient estimates, and it leads to computational
procedures closely similar to those of IPA. In this paper, based on this
approach, a rather general queueing network model with parameter-dependent
service times and routing mechanism is presented. For the performance
measures which one normally chooses in analysis of network performance, we
propose a gradient estimate, and give sufficient conditions for the estimate
to be unbiased. These conditions are rather general and normally met in
analysis of queueing network performance. Finally, an algorithm of estimating
gradient of a particular performance measure is presented, and its validity is
briefly discussed.

\section{The Underlying Network Model}

We consider a generalized model of a queueing network consisting of $ N $
nodes, with customers of a single class. As is customary in queueing network
models, customers are assumed to circulate through the network to receive
service at appropriate nodes. We do not restrict ourselves to a particular
type of nodes, it is suggested that any node may have a single server as well
as several servers operating either in parallel or in tandem.

Furthermore, there is a buffer with infinite capacity in each node, in which
customers are placed at their arrival to wait for service if it cannot be
initiated immediately. We assume the queue discipline underlying the operation
of any node to be first-come, first-served. Upon his service completion at one
node, each customer goes to another node chosen according to some routing
procedure described below. We suppose that the transition of any customer
between nodes requires no time, and he therefore arrives immediately into the
next node. Finally, we assume that the network starts operating at time zero;
at the initial time, the server at any node $ n $ is free, whereas its
buffer contains $ K_{n} $ customers, $ 0 \leq K_{n} \leq \infty $,
$ n=1,\ldots,N $.

We now turn to the formal description of the network dynamics from an
algebraic viewpoint, and then introduce randomness into the network model.

\subsection{Algebraic Description of Node Dynamics}

In a general sense, each node can be regarded as a processor which produces an
output sequence of departure times of customers from another two, an input and
control sequences formed respectively by the arrival times and the service
times of customers. Let us denote for every node $ n $,
$ 1 \leq n \leq N $, the $k$th arrival epoch to the node by $ A_{n}^{k} $,
and the $k$th departure epoch from the node by $ D_{n}^{k} $. Notice,
because the transition of customers from one node to another is immediate,
each $ A_{n}^{k} $ coincides with some $ D_{i}^{j} $ with the
exception of $ A_{n}^{k} =  0 $ for all $ k \leq K_{n} $. Finally, we
denote the service time associated with the $k$th service initiation in node
$ n $, by $ \tau_{n}^{k} $. The set of all service times
$ \mbox{\boldmath $T$}=\{\tau_{n}^{k} |  n=1,\ldots,N; k=1,2,\ldots \} $
is assumed to be given.

The usual way to represent the operation of a node is based on recursive
equations describing evolution of $ D_{n}^{k} $ as a state variable
\cite{Chen90,HuJQ92,Kriv94b}. Note that these recursive equations are often
rather difficult to resolve. Below are given two equations which describe
dynamics of nodes operating as the $ G/G/1 $ and $ G/G/2 $ queueing
systems. Other examples may be found in \cite{Chen90,HuJQ92,Kriv94a,Kriv94b}.

\subsubsection{The $G/G/1$ queue.}
Suppose first that node $ n $ is represented as the $ G/G/1 $ queue.
Its associated recursive equation may be written as \cite{Chen90}
$$
D_{n}^{k} = (A_{n}^{k} \vee D_{n}^{k-1}) + \tau_{n}^{k},
$$
where $ \vee $ denotes the maximum operator, and
$ D_{n}^{k} \equiv 0 $ for all $ k < 0 $. It is easy to see that the
solution of the equation in terms of arrival and service times has the form
$$
D_{n}^{k}
    = \bigvee_{i=1}^{k} \left(A_{n}^{i} + \sum_{j=i}^{k} \tau_{n}^{j} \right).
$$

\subsubsection{The $G/G/2$ queue.}
The equation which describes the dynamics of a node operating as the
$ G/G/2 $ queue may be considered as rather difficult to handle. For node
$ n $, it is written as \cite{Kriv94a}
$$
D_{n}^{k} =
 \bigvee_{i=1}^{k} \left((A_{n}^{i}\vee D_{n}^{i-2}) + \tau_{n}^{i} \right)
 \wedge \left((A_{n}^{k+1}\vee D_{n}^{k-1}) + \tau_{n}^{k+1} \right),
$$
where $ \wedge $ stands for the minimum operator. Although there are no
closed-form solutions of the equation, known to the author, it is clear that
it exists.

\subsection{Routing Mechanism and Interaction of Nodes}

The routing mechanism inherent in the network is defined by the sequences
$ \mbox{\boldmath $R$}_{n} = \{\rho_{n}^{1}, \rho_{n}^{2}, \ldots\} $
given for each node $ n $, where $ \rho_{n}^{k} $ represents the next
node to be visited by the customer who is the $k$th to depart from node
$ n $, $ \rho_{n}^{k} \in \{1, \ldots, N\} $, $ k=1,2, \ldots $. The
matrix
$$
\mbox{\boldmath $R$} =
\left( \begin{array}{ccccc}
         \rho_{1}^{1} & \rho_{1}^{2} & \ldots & \rho_{1}^{k} & \ldots \\
         \rho_{2}^{1} & \rho_{2}^{2} & \ldots & \rho_{2}^{k} & \ldots \\
         \vdots         & \vdots         &        & \vdots         &        \\
         \rho_{N}^{1} & \rho_{N}^{2} & \ldots & \rho_{N}^{k} & \ldots \\
       \end{array}
\right)
$$
is referred to as the routing table of the network.

In order to describe the dynamics of the network completely, it remains to
define formally interactions between nodes. In fact, a relationship between
arrival and departure times of distinct nodes is to be established. To this
end, for each node $ n $, let us introduce the set
\begin{equation} \label{D-def}
\mbox{\boldmath $D$}_{n}
= \{ D_{i}^{j} |  \rho_{i}^{j} = n;  i=1,\ldots,N;  j=1,2,\ldots \}
\end{equation}
which is constituted by the departure times of the customers who have to go to
node $ n $. Furthermore, we denote by $ {\cal A}_{n}^{k} $ the arrival
time of the customer which is the $k$th to arrive into node $ n $ after
his service at any node of the network. In other words, the symbol
$ {\cal A}_{n}^{k} $ differs from $ A_{n}^{k} $ in that it refers only
to the customers really arriving into node $ n $, and does not to those
occurring in this node at the initial time.

It has been shown in \cite{Kriv90b,Kriv93} that it holds
\begin{equation} \label{A1-def}
{\cal A}_{n}^{k} =
\bigwedge_{\{D_{1},\ldots,D_{k}\} \subset \bm{D}_{n}}
                              (D_{1} \vee \cdots \vee D_{k}),
\end{equation}
where minimum is taken over all $k$-subsets of the set
$ \mbox{\boldmath $D$}_{n} $. The times $ A_{n}^{k} $ and
$ {\cal A}_{n}^{k} $ are related by the equality
\begin{equation} \label{A-def}
A_{n}^{k} = \left\{\begin{array}{ll}
		       0, & \mbox{if $ k \leq K_{n} $}  \\
		       {\cal A}_{n}^{k-K_{n}}, & \mbox{otherwise}
		   \end{array}
            \right..
\end{equation}

Clearly, if deterministic routing with an integer matrix $ R $ as the
routing table is adopted in the model, each set
$ \mbox{\boldmath $D$}_{n} $, $ n=1,\ldots,N $, is determined uniquely
from (\ref{D-def}), and then a straightforward algebraic representation of
$ A_{n}^{k} $ may be obtained using (\ref{A1-def}-\ref{A-def}). In this
case, starting from the above representations of node dynamics, one may
eventually arrive at algebraic expressions for any arrival time
$ A_{n}^{k} $ and departure time $ D_{n}^{k} $, which are written in
terms of service times $ \tau \in \mbox{\boldmath $T$} $, and involve only
the operations of maximum, minimum, and addition.

\subsection{Representation of Network Performance}

One of the features of the formal network model described above is that it
offers the potential for representing network performance criteria in a rather
simple and convenient way. Suppose that we observe the network until the $K$th
service completion at node $ n $, $ 1 \leq n \leq N $. As performance
criteria for node $ n $ in the observation period, one normally chooses
the following average quantities \cite{Chen90,Kriv90a,Kriv90b,Kriv93,Kriv94b}:
$$\begin{array}{ll}
\mbox{System time of one customer,} &
S_{n}^{K} = \sum_{k=1}^{K} (D_{n}^{k}-A_{n}^{k})/K, \\ \\
\mbox{Waiting time of one customer,} &
W_{n}^{K} = \sum_{k=1}^{K} (D_{n}^{k}-A_{n}^{k}-\tau_{n}^{k})/K, \\ \\
\mbox{Throughput rate of the node,} &
T_{n}^{K} = K/D_{n}^{K}, \\ \\
\mbox{Utilization of the server,} &
U_{n}^{K} = \sum_{k=1}^{K} \tau_{n}^{k}/D_{n}^{K}, \\ \\
\mbox{Number of customers,} &
J_{n}^{K} = \sum_{k=1}^{K} (D_{n}^{k}-A_{n}^{k})/D_{n}^{K}, \\ \\
\mbox{Queue length at the node,} &
Q_{n}^{K} = \sum_{k=1}^{K} (D_{n}^{k}-A_{n}^{k}-\tau_{n}^{k})/D_{n}^{K}.
\end{array}$$
It is easy to see that with the routing mechanism determined by an integer
matrix, all these criteria may be represented only in terms of service times
in closed form.

\subsection{Stochastic Aspect and Performance Evaluation}

Let us suppose that for all $ n=1,\ldots,N $, and $ k=1,2,\ldots $, the
service times are defined as random variables
$ \tau_{n}^{k} = \tau_{n}^{k}(\theta,\omega) $, where
$ \theta \in \Theta \subset \mathbb{R} $ is a decision parameter,
$ \omega $ is a random vector which represents the random effects involved
in network behaviour. First we assume the routing table
$ \mbox{\boldmath $R$} = R $ to be an integer matrix. Since deterministic
routing leads to algebraic expressions in terms of the random variables
$ \tau \in \mbox{\boldmath $T$} $ for the performance criteria introduced
above, one can conclude that these criteria also present random variables.

Let $ F=F(\theta,\omega) $ be a random performance criterion of the
network. As is customary, we define the performance measure associated with
$ F $ as the expected value
\begin{equation} \label{E-def}
\mbox{\boldmath $F$}(\theta) = E_{\omega}[F(\theta,\omega)].
\end{equation}
Although we may express $ F $ in closed form, it is often very difficult
or impossible to obtain analytically the performance measure
$ \mbox{\boldmath $F$} $. In this situation, one generally applies a
simulation technique which allows of obtaining values of
$ F(\theta,\omega) $, and then estimate the network performance by using the
Monte Carlo approach.

We now turn to the networks with parameter-dependent probabilistic routing. We
assume $ \rho_{n}^{k} = \rho_{n}^{k}(\theta,\omega) $ to be a discrete
random variable ranging over the set $ \{1,\ldots,N\} $. The routing
mechanism of the network is now defined by the random matrix
$ \mbox{\boldmath $R$} = \mbox{\boldmath $R$}(\theta,\omega) $ with
particular routing tables as its values. We denote the set of all possible
routing tables $ R $ by $ {\cal R} $.

Obviously, the expression of $ A_{n}^{k} $ defined by
(\ref{A1-def}-\ref{A-def}) may change from one shape into another, depending
on particular routing tables. To take this into account, we now define the
random performance criteria in (\ref{E-def}) as \cite{Kriv90b,Kriv93}
\begin{equation} \label{F-def}
F(\theta,\omega,\mbox{\boldmath $R$}(\theta,\omega)) =
\sum_{R\in{\cal R}} {\bf 1}_{\{\bm{R}(\theta,\omega) = R\}}
                                                         F_{R}(\theta,\omega),
\end{equation}
where $ {\bf 1}_{\{\bm{R}(\theta,\omega)=R\}} $ is the indicator
function of the event $ \{\mbox{\boldmath $R$}(\theta,\omega) = R\} $, and
$ F_{R}(\theta,\omega) = F(\theta,\omega,R) $ is the performance criterion
evaluated under the condition that the network operates according to the
deterministic routing mechanism defined by the routing table
$ R \in {\cal R} $.

\section{Performance Measure Gradient Estimation}

Since there are generally no explicit representations as functions of system
parameters $ \theta $, available for the performance measure
$ \mbox{\boldmath $F$} $, one may evaluate its gradient
$ \partial\mbox{\boldmath $F$}(\theta)/\partial\theta $ by no way other
than through the use of either finite difference estimates
\cite{CaoX85,HoYC87} or the estimate
\begin{equation} \label{g-def}
g(\theta,\omega_{1},\ldots,\omega_{M})
= \frac{1}{M} \sum_{i=1}^{M} \frac{\partial}{\partial\theta}
                                                         F(\theta,\omega_{i}),
\end{equation}
where $ \omega_{i} $, $ i=1,\ldots,M $, are independent realization of
$ \omega $, provided that the derivative
$ \partial F(\theta,\omega) / \partial\theta $ exists.

Very efficient procedures of obtaining gradient estimates may be designed
using the IPA technique \cite{HoYC83,HoYC87,Suri89}. Such a procedure can
yield the exact values of the derivative
$ \partial F(\theta,\omega) / \partial\theta $ by performing only one
simulation run. Furthermore, in the case of a vector parameter
$ \theta \in \mathbb{R}^{d} $, the IPA procedures provide all partial
derivatives $ \partial F(\theta,\omega) / \partial \theta_{i} $,
$ i=1,\ldots,d $, simultaneously, and take an additional computational cost
which is normally very small compared with that required for the simulation
run alone. Finally, it can be easily shown \cite{CaoX85,Suri89} that if the
IPA estimate of the derivative in (\ref{g-def}) is unbiased, the mean square
error of $ g $ has the order which is significantly less than those of any
finite difference estimates based on the same number of simulation runs.

A sufficient condition for the estimate (\ref{g-def}) to be unbiased at some
$ \theta \in \Theta $ requires \cite{CaoX85,HoYC87,Suri89}
\begin{equation} \label{C-def}
\frac{\partial}{\partial\theta} E[F(\theta,\omega)]
= E\left[ \frac{\partial}{\partial\theta} F(\theta,\omega) \right].
\end{equation}
A usual way of examining the interchange in (\ref{C-def}) involves the
application of the Lebesgue dominated convergence theorem \cite{Loev60}

\begin{theorem} \label{L-t}
Let $ (\Omega,{\cal F},P) $ be a probability space,
$ \Theta \subset \mathbb{R}^{d} $, and
$ F: \Theta \times \Omega \rightarrow \mathbb{R} $ be a
${\cal F}$-measurable function for any $ \theta \in \Theta $ and such that
the following conditions hold:
\begin{description}
\item{{\rm (i)}} for every $ \theta \in \Theta $, there exists
$ \partial F(\theta,\omega) / \partial\theta $ at $ \theta $ w.p.~1,
\item{{\rm (ii)}} for all $ \theta_{1}, \theta_{2} \in \Theta $, there is
a random variable $ \lambda(\omega) $ with $ E\lambda < \infty $ and
such that
\begin{equation} \label{L-def}
| F(\theta_{1},\omega)-F(\theta_{2},\omega) |
\leq \lambda(\omega) \parallel \theta_{1}-\theta_{2} \parallel \;\;\;
\mbox{w.p.~1.}
\end{equation}
\end{description}
Then equation {\rm (\ref{C-def})} holds on $ \Theta $.
\end{theorem}

In \cite{Kriv90a,Kriv90b,Kriv93} the approach based on the implementation of
Theorem~\ref{L-t} has been applied to analyze estimates of performance
gradient in the networks models with the parameter-independent probabilistic
routing mechanism determined by a random routing table
$ \mbox{\boldmath $R$} = \mbox{\boldmath $R$}(\omega) $. Specifically,
starting from the representations of network dynamics, discussed in Section~2,
it has been shown that
\begin{description}
\item{{\rm (i)}} if each service time $ \tau \in \mbox{\boldmath $T$} $
satisfies the conditions of Theorem~\ref{L-t}, and for every
$ \theta \in \Theta $, all $ \tau \in \mbox{\boldmath $T$} $ present
continuous and independent random variables, then the average total time
$ S_{n}^{K} $ and waiting time $ W_{n}^{K} $ satisfy the conditions of
Theorem~\ref{L-t};
\item{{\rm (ii)}} if in addition to previous assumptions, there exist random
variables $ \mu, \nu > 0 $ such that for all
$ \tau \in \mbox{\boldmath $T$} $ it holds $ \nu \leq |\tau| \leq\mu $
w.p.~1 for all $ \theta \in \Theta $, and the condition
$ E[\mu\lambda/\nu^{2}] < \infty $ is fulfilled, where $ \lambda $ is
the random variable providing $ \tau $ with (\ref{L-def}), then the
average throughput rate $ T_{n}^{K} $, utilization $ U_{n}^{K} $, number
of customers $ J_{n}^{K} $, and queue length $ Q_{n}^{K} $ satisfy the
conditions of Theorem~\ref{L-t}.
\end{description}
Note that the above conditions do not involve independence at each
$ \theta $ between the random variables $ \tau(\theta,\omega) $ and
$ \rho(\omega) $ in the probabilistic sense.

\section{An Unbiased Gradient Estimate for Networks with
Parameter-Dependent Routing}

We start the section with an example which exhibits difficulties arising in
gradient estimation when there is a parameter dependence involved in the
routing mechanism of the network, and then present our main result offering an
unbiased estimate of performance measure gradient in networks with
parameter-dependent routing.

\subsection{Preliminary Analysis}

Suppose that there is a parameter dependence of the routing mechanism in the
network, that is
$ \mbox{\boldmath $R$} = \mbox{\boldmath $R$}(\theta,\omega) $. In this
case, the random performance criterion $ F $ generally violates condition
(\ref{L-def}). As an illustration, one can consider the following example.

Let $ \Theta = [0,1] $, $ \Omega_{1}=\Omega_{2}=[0,1] $, and
$ (\Omega,{\cal F},P) $ be a probability space, where $ {\cal F} $ is
the $\sigma$-field of Borel sets of
$ \Omega = \Omega_{1} \times \Omega_{2} $, $ P $ is the Lebesgue measure
on $ \Omega $. Denote $ \omega = (\omega_{1},\omega_{2})^{T} $, and define
the function
$$
F(\theta,\omega,\rho(\theta,\omega))
= \left\{\begin{array}{ll}
         \tau_{1}(\theta,\omega), & \mbox{if $ \rho(\theta,\omega)=1 $} \\
         \tau_{2}(\theta,\omega), & \mbox{if $ \rho(\theta,\omega)=2 $}
		   \end{array},
            \right.
$$
where
$$
\tau_{1}(\theta,\omega) = \theta+\omega_{1}+1, \;\;\;
\tau_{2}(\theta,\omega) = \theta+\omega_{1},
$$
and
$$
\rho(\theta,\omega)
= \left\{\begin{array}{ll}
         1, & \mbox{if $ \omega_{2} \leq \theta $} \\
         2, & \mbox{if $ \omega_{2} > \theta $}
		   \end{array}
            \right..
$$
We may treat $ \tau_{1} $ and $ \tau_{2} $ as the service time of a
customer respectively at node $ 1 $ and $ 2 $. The function $ F $ is
then assumed to be the service time of the customer which may arrive into
either node $ 1 $ or $ 2 $, according to one of the two possible values
of $ \rho $.

Clearly, $ \tau_{1} $ and $ \tau_{2} $ satisfy the conditions of
Theorem~\ref{L-t}, whereas the function $ F $ now represented as
$$
F(\theta,\omega)
= \left\{\begin{array}{ll}
         \theta+\omega_{1}+1,   & \mbox{if $ \omega_{2} \leq \theta $} \\
         \theta+\omega_{1}, & \mbox{if $ \omega_{2} > \theta $}
		   \end{array}
            \right.,
$$
is differentiable w.p.~1 at any $ \theta \in \Theta $, and
$ \partial F(\theta,\omega) / \partial\theta = 1 $ w.p.~1. However, for
any $ \theta_{1},\theta_{2} \in \Theta $ such that
$ \theta_{1} \geq \omega_{2} $ and $ \theta_{2} < \omega_{2} $, it holds
$$
| F(\theta_{1},\omega)-F(\theta_{2},\omega) | \geq 1,
$$
and therefore, condition (\ref{L-def}) is violated.

On the other hand, it is easy to verify that
$$
\begin{array}{lcl}
E[F(\theta,\omega)] = 2\theta+\frac{1}{2}, & &
\frac{\partial}{\partial\theta} E[F(\theta,\omega)] = 2 \\ \\
\frac{\partial}{\partial\theta} F(\theta,\omega) = 1 \;\; \mbox{w.p.~1}, & &
E\left[\frac{\partial}{\partial\theta} F(\theta,\omega) \right] = 1.
\end{array}
$$
In other words, equation (\ref{C-def}) proves to be not valid, and we finally
conclude that estimate (\ref{g-def}) will be biased.

\subsection{The Main Result}

To suppress the bias in estimates of the gradient
\begin{equation} \label{gF-def}
\frac{\partial}{\partial\theta}\mbox{\boldmath $F$}(\theta)
= \frac{\partial}{\partial\theta}
E[F(\theta,\omega,\mbox{\boldmath $R$}(\theta,\omega))],
\end{equation}
let us replace (\ref{g-def}) by the estimate
\begin{equation} \label{h-def}
{\widetilde g}(\theta,\omega_{1},\ldots,\omega_{M})
                           = \frac{1}{M} \sum_{i=1}^{M} G(\theta,\omega_{i}),
\end{equation}
where $ G(\theta,\omega) $ will be defined in the next theorem.

\begin{theorem} \label{M-t}
Suppose that a random performance criterion $ F $ is represented in
form {\rm (\ref{F-def})}, and for each $ R \in {\cal R} $ the following
conditions hold:
\begin{description}
\item{{\rm (i)}} $ F_{R} $ satisfies the conditions of Theorem~\ref{L-t};
\item{{\rm (ii)}} for any $ \theta \in \Theta $, the random variable
$ F_{R}(\theta,\omega) $ and the random matrix
$ \mbox{\boldmath $R$}(\theta,\omega) $ are independent;
\item{{\rm (iii)}} for any $ \theta \in \Theta $, the function
$ \Phi(\theta,R)
= \mbox{\rm Pr}\{\mbox{\boldmath $R$}(\theta,\omega)=R\} $ is continuously
differentiable at $ \theta $, and $ \Phi(\theta,R) > 0 $.
\end{description}
Then for any $ \theta_{0} \in \Theta $, estimate {\rm (\ref{h-def})} with
$$
G(\theta_{0},\omega) = \left. \frac{\partial}{\partial\theta}
                      F(\theta,\omega,\mbox{\boldmath $R$}(\theta_{0},\omega))
                                                \right|_{\theta=\theta_{0}}
+ F(\theta_{0},\omega,\mbox{\boldmath $R$}(\theta_{0},\omega))
                     \Psi(\theta_{0},\mbox{\boldmath $R$}(\theta_{0},\omega)),
$$
where $ \Psi(\theta,R) = \frac{\partial}{\partial\theta}\ln\Phi(\theta,R) $,
is unbiased.
\end{theorem}
\begin{proof}
Clearly, it is sufficient to show that the equation
$$
\frac{\partial}{\partial\theta}
E[F(\theta,\omega,\mbox{\boldmath $R$}(\theta,\omega))]
= E[G(\theta,\omega)]
$$
holds for any $ \theta \in \Theta $.

To verify this equation, let us first represent $ F $ in form
(\ref{F-def}), and consider its expected value
$$
E[F(\theta,\omega,\mbox{\boldmath $R$}(\theta,\omega))]
= \sum_{R\in{\cal R}} E\left[{\bf 1}_{\{\bm{R}(\theta,\omega) = R\}}
                                                    F(\theta,\omega,R)\right].
$$
Since $ E[{\bf 1}_{\{\bm{R}(\theta,\omega) = R\}}]
= \mbox{\rm Pr}\{\mbox{\boldmath $R$}(\theta,\omega) = R\}
= \Phi(\theta,R) $, it follows from condition~(ii) of the theorem that
\begin{eqnarray*}
\lefteqn{E[F(\theta,\omega,\mbox{\boldmath $R$}(\theta,\omega))]}
                                                       \;\;\;\;\;\;\;\;\;\; \\
& = & \sum_{R\in{\cal R}} E[F(\theta,\omega,R)]
  \mbox{\rm Pr}\{\mbox{\boldmath $R$}(\theta,\omega) = R\}
= \sum_{R\in{\cal R}} E[F(\theta,\omega,R)] \Phi(\theta,R).
\end{eqnarray*}

For any $ \theta_{0} \in \Theta $, under conditions (i) and (iii), we
successively get
\begin{eqnarray*}
\lefteqn{\left. \frac{\partial}{\partial\theta}
E[F(\theta,\omega,\mbox{\boldmath $R$}(\theta,\omega))]
         \right|_{\theta=\theta_{0}}}                                       \\
& = & \sum_{R\in{\cal R}} \left( \left. \frac{\partial}{\partial\theta}
     E[F(\theta,\omega,R)]\right|_{\theta=\theta_{0}} \Phi(\theta_{0},R)
                                                                    \right. \\
&   & \left. \mbox{} + E[F(\theta_{0},\omega,R)] \left.
                                \frac{\partial}{\partial\theta} \Phi(\theta,R)
                                 \right|_{\theta=\theta_{0}} \right)        \\
& = &  \sum_{R\in{\cal R}} \left( E\left[\left. \frac{\partial}{\partial\theta}
  F(\theta,\omega,R)]\right|_{\theta=\theta_{0}} \right] \Phi(\theta_{0},R)
                                                                    \right. \\
&   & \left. \mbox{} + E[F(\theta_{0},\omega,R)]
                   \frac{\left. \frac{\partial}{\partial\theta} \Phi(\theta,R)
                         \right|_{\theta=\theta_{0}}}
                            {\Phi(\theta_{0},R)} \Phi(\theta_{0},R) \right) \\
& = & \sum_{R\in{\cal R}} \Biggl( E\left[ \left. \frac{\partial}{\partial\theta}
                      F(\theta,\omega,R)\right|_{\theta=\theta_{0}} \right]
+ E[F(\theta_{0},\omega,R)] \Psi(\theta_{0},R) \Biggr)
                                                         \Phi(\theta_{0},R) \\
& = & \sum_{R\in{\cal R}} E\Biggl[ \left. \frac{\partial}{\partial\theta}
                               (\theta,\omega,R)\right|_{\theta=\theta_{0}}
+ F(\theta_{0},\omega,R) \Psi(\theta_{0},R) \Biggr]
             \mbox{\rm Pr}\{\mbox{\boldmath $R$}(\theta_{0},\omega) = R\} \\
& = & E \Biggl[ \left. \frac{\partial}{\partial\theta}
               F(\theta,\omega,\mbox{\boldmath $R$}(\theta_{0},\omega))
                                                \right|_{\theta=\theta_{0}}
+ F(\theta_{0},\omega,\mbox{\boldmath $R$}(\theta_{0},\omega))
        \Psi(\theta_{0},\mbox{\boldmath $R$}(\theta_{0},\omega)) \Biggr] \\
& = & E[G(\theta_{0},\omega)].
\qedhere
\end{eqnarray*}
\end{proof}

It is not difficult to obtain the conditions for estimate (\ref{h-def}) to be
unbiased for gradient of particular performance measures. They can be stated
by combining the conditions in Section~3, related to particular
performance criteria in networks with parameter-independent routing, with
those of Theorem~\ref{M-t}. Note that these conditions are rather general, and
normally met in analysis of queueing networks.

Let us now return to the example presented in the previous subsection. First,
we have
$$ \Phi(\theta,1) = \mbox{\rm Pr}\{\rho(\theta,\omega)=1\}=\theta, $$
$$ \Phi(\theta,2) = \mbox{\rm Pr}\{\rho(\theta,\omega)=2\}=1-\theta, $$
and then
$$ \Psi(\theta,1) = \frac{d}{d\theta}\ln\theta=\frac{1}{\theta}, $$
$$ \Psi(\theta,2) = \frac{d}{d\theta}\ln(1-\theta)=\frac{1}{1-\theta}. $$
In this case, the function $ G $ is defined as
$$
G(\theta,\omega)
= \left\{\begin{array}{ll}
                     1+\frac{\theta+\omega_{1}+1}{\theta}, &
                                  \mbox{if $ \omega_{2} \leq \theta $} \\ \\
                     1+\frac{\theta+\omega_{1}}{\theta-1}, &
                                        \mbox{if $ \omega_{2} > \theta $}
		   \end{array}
            \right.
$$
for any $ \theta \in (0,1) $. Finally, evaluation of its expected value
gives
$$
E[G(\theta,\omega)]
= \frac{\partial}{\partial\theta}
E\left[F(\theta,\omega,\rho(\theta,\omega))\right] = 2.
$$

\section{Application to Network Simulation}

Consider a network with $ N $ single-server nodes, and assume that the
$K$th service completion at a fixed node $ n $, $ 1 \leq n \leq N $, comes
with probability one after a finite number of service completions in the
network. In this case, to observe evolution of the network until the $K$th
completion at the node, it will suffice to take into consideration only finite
routes defined by a right truncated routing table with integer
($N\times L$)-matrices
$$
R = \left( \begin{array}{cccc}
             r_{1}^{1} & r_{1}^{2} & \ldots & r_{1}^{L} \\
             r_{2}^{1} & r_{2}^{2} & \ldots & r_{2}^{L} \\
             \vdots    & \vdots    &        & \vdots    \\
             r_{N}^{1} & r_{N}^{2} & \ldots & r_{N}^{L}
           \end{array}
    \right)
$$
as its values, with some $ L \geq K $.

Furthermore, we assume, as is customary in network simulation, that for each
$ \theta \in \Theta $, the random variables
$ \rho_{n}^{k}(\theta,\omega) $ are independent for all
$ n=1, \ldots, N $, and $ k=1, \ldots, L $. With this condition and the
notation $ \varphi_{n}^{k}(\theta,r) =
\mbox{\rm Pr}\{\rho_{n}^{k}(\theta,\omega) = r\} $, we may represent the
function $ \Phi $ introduced in Theorem~\ref{M-t}, as
$$
\Phi(\theta,R) = \mbox{\rm Pr}\{\mbox{\boldmath $R$}(\theta,\omega) = R\}
=  \prod_{n=1}^{N} \prod_{k=1}^{L} \varphi_{n}^{k}(\theta,r_{n}^{k}),
$$
and then get the function $ \Psi $ in the form
$$
\Psi(\theta,R) = \frac{\partial}{\partial\theta}\ln\Phi(\theta,R)
= \sum_{n=1}^{N} \sum_{k=1}^{L}
\frac{\partial}{\partial\theta}\ln\varphi_{n}^{k}(\theta,r_{n}^{k}).
$$

Suppose now that we have to estimate the gradient of a performance measure,
say $ \mbox{\boldmath $U$}_{n}^{K}(\theta) = E[U_{n}^{K}(\theta,\omega)] $,
the expected value of the average utilization of the server at node $ n $.
It results from Theorem~\ref{M-t} that, as an unbiased estimate, the function
\begin{equation} \label{G-def}
G(\theta,\omega) = \frac{\partial}{\partial\theta} F(\theta,\omega,R)
+ F(\theta,\omega,R) \Psi(\theta,R),
\end{equation}
may be applied with $ R = \mbox{\boldmath $R$}(\theta,\omega) $, and
$$
F(\theta,\omega,R)
= \sum_{k=1}^{K} \tau_{n}^{k}(\theta,\omega) / D_{n}^{K}(\theta,\omega,R).
$$

It is not difficult to construct the next algorithm which produces the value
of $ G(\theta,\omega) $ for fixed
$ \theta \in \Theta \subset \mathbb{R} $, and $ \omega \in \Omega $,
provided that there is a network simulation procedure into which the algorithm
may be incorporated. It actually combines an IPA algorithm \cite{HoYC83} for
obtaining the derivative $ \partial F(\theta,\omega,R)/\partial\theta $
with additional computations according to (\ref{G-def}).

\begin{algorithm}{5.1}
{\mdseries\itshape Initialization:} \\
for $ i=1,\ldots, N $ do $ g_{i} \longleftarrow 0 $; \\
$ s,t,t^{\prime} \longleftarrow 0 $; \\
$ R \longleftarrow \mbox{\boldmath $R$}(\theta,\omega) $; \\ \\
{\mdseries\itshape Upon the $k$th service completion at node $ i $, perform the instructions:} \\
$ g_{i} \longleftarrow g_{i}
           + \frac{\partial}{\partial\theta} \tau_{i}^{k}(\theta,\omega) $; \\
if \= $ i = n $ then \=
                    $ t \longleftarrow t + \tau_{i}^{k}(\theta,\omega) $; \\
   \> \>  $ t^{\prime} \longleftarrow t^{\prime}
           + \frac{\partial}{\partial\theta} \tau_{i}^{k}(\theta,\omega) $; \\
   \> \> if $ k = K $ then \=
                         $ d \longleftarrow D_{n}^{K}(\theta,\omega,R) $; \\
   \> \> \>              stop; \\
$ r \longleftarrow r_{i}^{k} $; \\
$ s \longleftarrow s
           + \frac{\partial}{\partial\theta}\ln\varphi_{i}^{k}(\theta,r) $; \\
if {\mdseries\itshape the server of node $ r $ is free} then
$ g_{r} \longleftarrow g_{i} $.
\end{algorithm}

On completion of the algorithm, it remains to compute
$ (t^{\prime}d - t g_{n})/d^{2} + t s/d $ as the value of $ G $.

Note, in conclusion, that estimate (\ref{g-def}) with the function $ G $
evaluated using the algorithm will not be unbiased in general. For the
estimate to be unbiased, the function $ \Psi(\theta,R) $ in (\ref{G-def})
must be calculated as the sum with the same number of summands
$ \partial \ln\varphi_{n}^{k}(\theta,r)/\partial\theta $ for any of 
simulation runs. However, during the simulation runs with distinct 
realizations of $ \omega $, there may be different numbers of service 
completions encountered at nodes $ i \neq n $, and then considered in 
evaluation of $ \Psi $. This normally involves an insignificant error in 
estimating the gradient, which becomes inessential as $ K $ increases.

\section{Acknowledgements}

The research described in this publication was made possible in part by Grant
\# NWA000 from the International Science Foundation.

\bibliographystyle{utphys}

\bibliography{Unbiased_gradient_estimation_in_queueing_networks_with_parameter-dependent_routing}

\end{document}